\documentclass[a4paper,10pt]{amsart}

\usepackage{amsmath,amssymb,amsfonts,latexsym}
\usepackage{enumitem}

\newtheorem{theorem}{Theorem}[section]
\newtheorem{lemma}[theorem]{Lemma}

\begin{document}

\title[Tribonacci-type octonions sequences]{The unifying formula for all Tribonacci-type octonions sequences and their properties}

\author[G. Cerda-Morales]{Gamaliel Cerda-Morales}
\address{Instituto de Matem\'aticas, Pontificia Universidad Cat\'olica de Valpara\'iso, Blanco Viel 596, Cerro Bar\'on, Valpara\'iso, Chile.}
\email{gamaliel.cerda@mail.pucv.cl}


\begin{abstract}
Various families of octonion number sequences (such as Fibonacci octonion, Pell octonion and Jacobsthal octonion) have been established by a number of authors in many different ways. In addition, formulas and identities involving these number sequences have been presented. In this paper, we aim at establishing new classes of octonion numbers associated with the generalized Tribonacci numbers. We introduce the Tribonacci and generalized Tribonacci octonions (such as Narayana octonion, Padovan octonion and third-order Jacobsthal octonion) and give some of their properties. We derive the relations between generalized Tribonacci numbers and Tribonacci octonions.

\vspace{2mm}

\vspace{2mm}


\end{abstract}

\thanks{This work was partially supported by IMA- Pontificia Universidad Cat\'olica de Valpara\'iso.}


\maketitle


\section{Introduction}
Recently, the topic of number sequences in real normed division algebras has attracted the attention of several researchers. It is worth noticing that there are exactly four real normed division algebras: real numbers ($\mathbb{R}$), complex numbers ($\mathbb{C}$), quaternions ($\mathbb{H}$) and octonions ($\mathbb{O}$). In \cite{Bae} Baez gives a comprehensive discussion of these algebras.

The real quaternion algebra $$\mathbb{H}=\{q=q_{r}+q_{i}\textbf{i}+q_{j}\textbf{j}+q_{k}\textbf{k}:\ q_{r},q_{l}\in \mathbb{R},\ l=i,j,k\}$$ is a 4-dimensional $\mathbb{R}$-vector space with basis $\{\textbf{1}\simeq e_{0},\textbf{i}\simeq e_{1},\textbf{j}\simeq e_{2},\textbf{k}\simeq e_{3}\}$ satisfying multiplication rules $q_{r}\textbf{1}=q_{r}$, $e_{1}e_{2}=-e_{2}e_{1}=e_{3}$, $e_{2}e_{3}=-e_{3}e_{2}=e_{1}$ and $e_{3}e_{1}=-e_{1}e_{3}=e_{2}$. Furthermore, the real octonion algebra denoted by $\mathbb{O}$ is an 8-dimensional real linear space with basis
\begin{equation}\label{eq:0}
\{e_{0}=\textbf{1}, e_{1}=\textbf{i}, e_{2}=\textbf{j}, e_{3}=\textbf{k}, e_{4}=\textbf{e}, e_{5}=\textbf{ie}, e_{6}=\textbf{je}, e_{7}=\textbf{ke}\},
\end{equation}
where $e_{0}\cdot e_{l}=e_{l}$ ($l=1,...,7$) and $q_{r}e_{0}=q_{r}$ ($q_{r}\in \mathbb{R}$).
The space $\mathbb{O}$ becomes an algebra via multiplication rules listed in the table \ref{table:1}, see \cite{Ta}.

\begin{table}[ht] 
\caption{The multiplication table for the basis of $\mathbb{O}$.} 
\centering      
\begin{tabular}{llllllll}
\hline
$\times $ & $e_{1}$ & $e_{2}$ & $e_{3}$ & $e_{4}$ & $e_{5}$ & $e_{6}$ & $%
e_{7}$ \\ \hline
$e_{1}$ & $-1$ & $e_{3}$ & $-e_{2}$ & $e_{5}$ & $-e_{4}$ & $-e_{7}$ & $e_{6}$
\\ 
$e_{2}$ & $-e_{3}$ & $-1$ & $e_{1}$ & $e_{6}$ & $e_{7}$ & $-e_{4}$ & $-e_{5}$
\\ 
$e_{3}$ & $e_{2}$ & $-e_{1}$ & $-1$ & $e_{7}$ & $-e_{6}$ & $e_{5}$ & $-e_{4}$
\\ 
$e_{4}$ & $-e_{5}$ & $-e_{6}$ & $-e_{7}$ & $-1$ & $e_{1}$ & $e_{2}$ & $e_{3}$
\\ 
$e_{5}$ & $e_{4}$ & $-e_{7}$ & $e_{6}$ & $-e_{1}$ & $-1$ & $-e_{3}$ & $e_{2}$
\\ 
$e_{6}$ & $e_{7}$ & $e_{4}$ & $-e_{5}$ & $-e_{2}$ & $e_{3}$ & $-1$ & $-e_{1}$
\\ 
$e_{7}$ & $-e_{6}$ & $e_{5}$ & $e_{4}$ & $-e_{3}$ & $-e_{2}$ & $e_{1}$ & $-1$
\\ \hline
\end{tabular}
\label{table:1}  
\end{table}

A variety of new results on Fibonacci-like quaternion and octonion numbers can be found in several papers \cite{Cer,Cim1,Cim2,Hal1,Hal2,Hor1,Hor2,Iye,Ke-Ak,Szy-Wl}. The origin of the topic of number sequences in division algebra can be traced back to the works by Horadam in \cite{Hor1} and Iyer in \cite{Iye}. In this sense, A. F. Horadam \cite{Hor1} defined the quaternions with the classic Fibonacci and Lucas number components as
\[
QF_{n}=F_{n}+F_{n+1}\textbf{i}+F_{n+2}\textbf{j}+F_{n+3}\textbf{k}
\]
and
\[
QL_{n}=L_{n}+L_{n+1}\textbf{i}+L_{n+2}\textbf{j}+L_{n+3}\textbf{k},
\]
respectively, where $F_{n}$ and $L_{n}$ are the $n$-th classic Fibonacci and Lucas numbers, respectively, and the author studied the properties of these quaternions. Several interesting and useful extensions of many of the familiar quaternion
numbers (such as the Fibonacci and Lucas quaternions \cite{Aky,Hal1,Hor1}, Pell quaternion \cite{Ca,Cim1}, Jacobsthal quaternions \cite{Szy-Wl} and third order Jacobsthal quaternion \cite{Cer}) have been considered by several authors. For example, in \cite{Cer1} a new type of quaternion whose coefficients are generalized Tribonacci numbers are defined.

There has been an increasing interest on quaternions and octonions that play an important role in various areas such as computer sciences, physics, differential geometry, quantum physics, signal, color image processing and geostatics (for more, see \cite{Ad,Car,Go1,Go2,Ko1,Ko2}).

In this paper, we define a family of the octonions, where the coefficients in the terms of the octonions are determined by the generalized Tribonacci numbers. These family of the octonions are called as the generalized Tribonacci octonions. Furthermore, we mention some of their properties, and apply them to the study of some identities and formulas of the generalized Tribonacci octonions. 

Here, our approach for obtaining some fundamental properties and characteristics of generalized Tribonacci octonions is to apply the properties of the generalized numbers introduced by Shannon and Horadam \cite{Sha}, Yalavigi \cite{Ya} and Pethe \cite{Pe}. This approach was originally proposed by Horadam and Iyer in the articles \cite{Hor1,Iye} for Fibonacci quaternions. The methods used by Horadam and Iyer in that papers have been recently applied to the other familiar octonion numbers by several authors \cite{Ak,Ca,Cer1,Cim2,Ke-Ak}. 

This paper has three main sections. In Section \ref{sect:2}, we provide the basic definitions of the octonions and the generalized Tribonacci numbers. Section \ref{sect:3} is devoted to introducing generalized Tribonacci octonions, and then to obtaining some fundamental properties and characteristics of these numbers. 

\section{Preliminaries}\label{sect:2}
We consider the generalized Tribonacci sequence $\{V_{n}(V_{0},V_{1},V_{2};r,s,t)\}$, or briefly $\{V_{n}\}$, defined as 
\begin{equation}\label{equ:1}
V_{n}=rV_{n-1}+sV_{n-2}+tV_{n-3},\ (n\geq 3),
\end{equation}
where $V_{0}$, $V_{1}$, $V_{2}$ are arbitrary integers and $r$, $s$, $t$, are real numbers. This sequence has been studied by Shannon and Horadam \cite{Sha}, Yalavigi \cite{Ya} and Pethe \cite{Pe}. If we set $r=s=t=1$ and $V_{0}=0$, $V_{1}=V_{2}=1$, then $\{V_{n}\}$ is the well-known Tribonacci sequence which has been considered extensively (see, for example, \cite{Fe}). 

As the elements of this Tribonacci-type number sequence provide third order iterative relation, its characteristic equation is $x^{3}-rx^{2}-sx-t=0$, whose roots are $\alpha=\alpha(r,s,t)=\frac{r}{3}+A_{V}+B_{V}$, $\omega_{1}=\frac{r}{3}+\epsilon A_{V}+\epsilon^{2} B_{V}$ and $\omega_{2}=\frac{r}{3}+\epsilon^{2}A_{V}+\epsilon B_{V}$, where $$A_{V}=\sqrt[3]{\frac{r^{3}}{27}+\frac{rs}{6}+\frac{t}{2}+\sqrt{\Delta}},\ B_{V}=\sqrt[3]{\frac{r^{3}}{27}+\frac{rs}{6}+\frac{t}{2}-\sqrt{\Delta}},$$ with $\Delta=\Delta(r,s,t)=\frac{r^{3}t}{27}-\frac{r^{2}s^{2}}{108}+\frac{rst}{6}-\frac{s^{3}}{27}+\frac{t^{2}}{4}$ and $\epsilon=-\frac{1}{2}+\frac{i\sqrt{3}}{2}$. 

In this paper, $\Delta(r,s,t)>0$, then the cubic equation $x^{3}-rx^{2}-sx-t=0$ has one real and two nonreal solutions, the latter being conjugate complex. Thus, the Binet formula for the generalized Tribonacci numbers can be expressed as:
\begin{equation}\label{eq:8}
V_{n}=\frac{P\alpha^{n}}{(\alpha-\omega_{1})(\alpha-\omega_{2})}-\frac{Q\omega_{1}^{n}}{(\alpha-\omega_{1})(\omega_{1}-\omega_{2})}+\frac{R\omega_{2}^{n}}{(\alpha-\omega_{2})(\omega_{1}-\omega_{2})},
\end{equation}
where $P=V_{2}-(\omega_{1}+\omega_{2})V_{1}+\omega_{1}\omega_{2}V_{0}$, $Q=V_{2}-(\alpha+\omega_{2})V_{1}+\alpha\omega_{2}V_{0}$ and $R=V_{2}-(\alpha+\omega_{1})V_{1}+\alpha\omega_{1}V_{0}$.

In fact, the generalized Tribonacci sequence is the generalization of the well-known sequences like Tribonacci, Padovan, Narayana and third-order Jacobsthal. For example, $\{V_{n}(0,1,1;1,1,1)\}_{n\geq0}$, $\{V_{n}(0,1,0;0,1,1)\}_{n\geq0}$, are Tribonacci and Padovan sequences, respectively. The Binet formula for the generalized Tribonacci sequence is expressed as follows:
\begin{lemma}
The Binet formula for the generalized Tribonacci sequence $\{V_{n}\}_{n\geq0}$ is:
\begin{equation}\label{eq:9}
V_{n+1}=V_{2}U_{n}+(sV_{1}+tV_{0})U_{n-1}+tV_{1}U_{n-2},
\end{equation}
and
\begin{equation}\label{eq:10}
U_{n}=\frac{\alpha^{n+1}}{(\alpha-\omega_{1})(\alpha-\omega_{2})}-\frac{\omega_{1}^{n+1}}{(\alpha-\omega_{1})(\omega_{1}-\omega_{2})}+\frac{\omega_{2}^{n+1}}{(\alpha-\omega_{2})(\omega_{1}-\omega_{2})},
\end{equation}
where $\alpha$, $\omega_{1}$ and $\omega_{2}$ are the roots of the cubic equation $x^{3}-rx^{2}-sx-t=0$.
\end{lemma}
\begin{proof}
The validity of this formula can be confirmed using the recurrence relation. Furthermore, $\{U_{n}\}_{n\geq0}=\{V_{n}(0,1,r;r,s,t)\}_{n\geq0}$.
\end{proof}

In the following we will study the important properties of the octonions. We refer to \cite{Bae} for a detailed analysis of the properties of the next octonions $p=\sum_{l=0}^{7}a_{l}e_{l}$ and $q=\sum_{l=0}^{7}b_{l}e_{l}$ where the coefficients $a_{l}, b_{l}\in \mathbb{R}$. We recall here only the following facts
\begin{itemize}[noitemsep]
\item The sum and subtract of $p$ and $q$ is defined as 
\begin{equation}\label{s1}
p\pm q=\sum_{l=0}^{7}(a_{l}\pm b_{l})e_{l},
\end{equation}
where $p\in \mathbb{O}$ can be written as $p=R_{p}+I_{p}$, and $R_{p}=a_{0}$ and $\sum_{l=1}^{7}a_{l}e_{l}$ are called the real and imaginary parts, respectively.
\item The conjugate of $p$ is defined by 
\begin{equation}\label{s2}
\overline{p}=R_{p}-I_{p}=a_{0}-\sum_{l=1}^{7}a_{l}e_{l}
\end{equation}
and this operation satisfies $\overline{\overline{p}}=p$, $\overline{p+q}=\overline{p}+\overline{q}$ and $\overline{p \cdot q}=\overline{q}\cdot \overline{p}$, for all $p,q\in \mathbb{O}$.
\item The norm of an octonion, which agrees with the standard Euclidean norm on $\mathbb{R}^{8}$ is defined as 
\begin{equation}\label{s3}
Nr^{2}(p)=\overline{p}\cdot p=p\cdot \overline{p}=\sum_{l=0}^{7}a_{l}^{2},\ \left(Nr(p)=\sqrt{\sum_{l=0}^{7}a_{l}^{2}}\in \mathbb{R}^{+}_{0}\right).
\end{equation}
\item The inverse of $p\neq 0$ is given by $p^{-1}=\frac{\overline{p}}{Nr^{2}(p)}$. From the above two definitions it is deduced that 
\begin{equation}\label{s4}
Nr^{2}(p\cdot q)=Nr^{2}(p)Nr^{2}(q)\ \textrm{and}\ (p\cdot q)^{-1}=q^{-1}\cdot p^{-1}.
\end{equation}
\item $\mathbb{O}$ is non-commutative and non-associative but it is alternative, in other words 
\begin{equation}\label{s5}
\begin{aligned}
p\cdot(p\cdot q)&=p^{2}\cdot q,\\
(p\cdot q)\cdot q&=p\cdot q^{2},\\
(p\cdot q)\cdot p&=p\cdot (q\cdot p)=p\cdot q\cdot p,
\end{aligned}
\end{equation}
where $\cdot$ denotes the product in the octonion algebra $\mathbb{O}$.
\end{itemize}

\section{The Generalized Tribonacci Octonions}\label{sect:3}
In this section, we define new kinds of sequences of octonion number called as generalized Tribonacci octonions. We study some properties of these octonions. We obtain various results for these classes of octonion numbers included recurrence relations, summation formulas, Binet's formulas and generating functions.

In \cite{Cer1}, the author introduced the so-called generalized Tribonacci quaternions, which are a new class of quaternion sequences. They are defined by
\begin{equation}\label{eq:1}
Q_{v,n}=\sum_{l=0}^{3}V_{n+l}e_{l}=V_{n}+\sum_{l=1}^{3}V_{n+l}e_{l},\ (V_{n}\textbf{1}=V_{n}),
\end{equation}
where $V_{n}$ is the $n$-th generalized Tribonacci number, $e_{1}^{2}=e_{2}^{2}=e_{3}^{2}=-\textbf{1}$ and $e_{1}e_{2}e_{3}=-\textbf{1}$.

We now consider the usual generalized Tribonacci numbers, and based on the definition (\ref{eq:1}) we give definition of a new kind of octonion numbers, which we call the generalized Tribonacci octonions. In this paper, we define the $n$-th generalized Tribonacci octonion number by the following recurrence relation
\begin{equation}
\begin{aligned}
O_{v,n}&=V_{n}+\sum_{l=1}^{7}V_{n+l}e_{l},\ n\geq 0\\
&=V_{n}+V_{n+1}e_{1}+V_{n+2}e_{2}+V_{n+3}e_{3}\\
&\ \ +V_{n+4}e_{4}+V_{n+5}e_{5}+V_{n+6}e_{6}+V_{n+7}e_{7},
\end{aligned}\label{eq:2}
\end{equation}
where $V_{n}$ is the $n$-th generalized Tribonacci number. Here $\{e_{l}:\ l=0,1,...,7\}$ satisfies the multiplication rule given in the Table \ref{table:1}. Furthermore, the sequence $\{U_{n}\}$ is the special case of $\{V_{n}\}$ where $V_{0}=0$, $V_{1}=1$ and $V_{2}=r$. Then, we can write $O_{u,n}=U_{n}+\sum_{l=1}^{7}U_{n+l}e_{l},\ (n\geq 0)$.

The equalities in (\ref{s1}) gives 
\begin{equation}\label{s6}
O_{v,n}\pm O_{v,m}=\sum_{l=0}^{7}(V_{n+l}\pm V_{m+l})e_{l}\ (n,m\geq 0).
\end{equation}
From (\ref{s2}), (\ref{s3}) and (\ref{eq:2}) an easy computation gives 
\begin{equation}\label{s7}
\overline{O_{v,n}}=V_{n}-\sum_{l=1}^{7}V_{n+l}e_{l},\ \textrm{and}\ Nr(O_{v,n})=\sqrt{\sum_{l=0}^{7}V_{n+l}^{2}}\in \mathbb{R}^{+}_{0}.
\end{equation}

By some elementary calculations we find the following recurrence relation for the generalized Tribonacci octonions from (\ref{eq:2}), (\ref{s6}) and (\ref{equ:1}):
\begin{equation}
\begin{aligned}
rO_{v,n+1}+sO_{nv,}+tO_{v,n-1}&=\sum_{l=0}^{7}(rV_{n+l+1}+sV_{n+l}+tV_{n+l-1})e_{l}\\
&=V_{n+2}+\sum_{l=1}^{7}V_{n+l+2}e_{l}\\
&=O_{v,n+2}\ \  (n\geq 1).
\end{aligned} \label{equ:3}
\end{equation}

Now, we will state Binet's formula for the generalized Tribonacci octonions. Repeated use of (\ref{eq:8}) in (\ref{eq:2}) enables one to write for $\underline{\alpha}=\sum_{l=0}^{7}\alpha^{l}e_{l}$, $\underline{\omega_{1}}=\sum_{l=0}^{7}\omega_{1}^{l}e_{l}$ and $\underline{\omega_{2}}=\sum_{l=0}^{7}\omega_{2}^{l}e_{l}$:
\begin{equation}
\begin{aligned}
O_{v,n}&=\sum_{l=0}^{7}V_{n+l}e_{l}\\
&=\sum_{l=0}^{7}\left(\frac{P\alpha^{n+l}e_{l}}{(\alpha-\omega_{1})(\alpha-\omega_{2})}-\frac{Q\omega_{1}^{n+l}e_{l}}{(\alpha-\omega_{1})(\omega_{1}-\omega_{2})}+\frac{R\omega_{2}^{n+l}e_{l}}{(\alpha-\omega_{2})(\omega_{1}-\omega_{2})}\right)\\
&=\frac{P\underline{\alpha}\alpha^{n}}{(\alpha-\omega_{1})(\alpha-\omega_{2})}-\frac{Q\underline{\omega_{1}}\omega_{1}^{n}}{(\alpha-\omega_{1})(\omega_{1}-\omega_{2})}+\frac{R\underline{\omega_{2}}\omega_{2}^{n}}{(\alpha-\omega_{2})(\omega_{1}-\omega_{2})},
\end{aligned} \label{equ:5}
\end{equation}
where $\alpha$, $\omega_{1}$ and $\omega_{2}$ are the roots of the cubic equation $x^{3}-rx^{2}-sx-t=0$, and $P$, $Q$ and $R$ as before. The formula in (\ref{equ:5}) is called as Binet's formula for the generalized Tribonacci octonions. 

In the following theorem we present the generating function for generalized Tribonacci octonions.
\begin{theorem}\label{n1}
The generating function for the generalized Tribonacci octonion $O_{v,n}$ is
\begin{equation}\label{p1}
g(x)=\frac{O_{v,0}+(O_{v,1}-rO_{v,0})x+(O_{v,2}-rO_{v,1}-sO_{v,0})x^{2}}{1-rx-sx^{2}-tx^{3}}.
\end{equation}
\end{theorem}
\begin{proof}
Assuming that the generating function of the generalized Tribonacci octonion sequence $\{O_{v,n}\}_{n\geq0}$ has the form $g(x)=\sum_{n\geq 0}O_{v,n}x^{n}$, we obtain that
\begin{align*}
(1-rx-sx^{2}&-tx^{3})\sum_{n\geq 0}O_{v,n}x^{n}\\
&=O_{v,0}+O_{v,1}x+O_{v,2}x^{2}+O_{v,3}x^{3}+\cdots\\
&\ \ -rO_{v,0}x-rO_{v,1}x^{2}-rO_{v,2}x^{3}-rO_{v,3}x^{4}-\cdots\\
&\ \ -sO_{v,0}x^{2}-sO_{v,1}x^{3}-sO_{v,2}x^{4}-sO_{v,3}x^{5}-\cdots\\
&\ \ -tO_{v,0}x^{3}-tO_{v,1}x^{4}-tO_{v,2}x^{5}-tO_{v,3}x^{6}-\cdots\\
&=O_{v,0}+(O_{v,1}-rO_{v,0})x+(O_{v,2}-rO_{v,1}-sO_{v,0})x^{2},
\end{align*}
since $O_{v,n}=rO_{v,n-1}+sO_{v,n-2}+tO_{v,n-3}$, $n\geq 3$ and the coefficients of $x^{n}$ for $n\geq 3$ are equal with zero. Then, we get $$g(x)=\frac{O_{v,0}+(O_{v,1}-rO_{v,0})x+(O_{v,2}-rO_{v,1}-sO_{v,0})x^{2}}{1-rx-sx^{2}-tx^{3}}.$$ The theorem is proved.
\end{proof}

In Table \ref{table:2}, we examine some special cases of generating functions given in Eq. (\ref{p1}).
\begin{table}[ht] 
\caption{Generating functions according to initial values.} 
\centering      
\begin{tabular}{ll}
\hline
$\textrm{Narayana}$ & $\frac{\left\lbrace \begin{array}{c}
x+e_{1}+(1+x^{2})e_{2}+(1+x+x^{2})e_{3}\\
+(2+x+x^{2})e_{4}+(3+x+2x^{2})e_{5}\\
+(4+2x+3x^{2})e_{6}+(6+3x+4x^{2})e_{7}\end{array} \right\rbrace}{1-x-x^{3}}$  \\ 
$\textrm{Tribonacci}$  & $\frac{\left\lbrace \begin{array}{c}
x+e_{1}+(1+x+x^{2})e_{2}+(2+2x+x^{2})e_{3}\\
+(4+3x+2x^{2})e_{4}+(7+6x+4x^{2})e_{5}\\
+(13+11x+7x^{2})e_{6}+(24+20x+13x^{2})e_{7}\end{array} \right\rbrace}{1-x-x^{2}-x^{3}}$ \\ 
$\textrm{Padovan}$  & $\frac{\left\lbrace \begin{array}{c}
x+e_{1}+(x+x^{2})e_{2}+(1+x)e_{3}\\
+(1+x+x^{2})e_{4}+(1+2x+x^{2})e_{5}\\
+(2+2x+x^{2})e_{6}+(2+3x+2x^{2})e_{7}\end{array} \right\rbrace}{1-x^{2}-x^{3}}$ \\ 
$\textrm{Third-Order Jacobsthal}$  & $\frac{\left\lbrace \begin{array}{c}
x+e_{1}+(1+x+x^{2})e_{2}+(2+3x+2x^{2})e_{3}\\
+(5+4x+4x^{2})e_{4}+(9+9x+10x^{2})e_{5}\\
+(18+19x+18x^{2})e_{6}+(37+36x+36x^{2})e_{7}\end{array} \right\rbrace}{1-x-x^{2}-2x^{3}}$
\\ \hline
\end{tabular}
\label{table:2}  
\end{table}

Now, let us write the formula which gives the summation of the first $n$ generalized Tribonacci numbers and octonions. 
\begin{lemma}[\cite{Cer1}]
For every integer $n\geq 0$, we have:
\begin{equation}\label{p2}
\sum_{l=0}^{n}V_{l}=\frac{1}{\delta_{r,s,t}}\left\lbrace \begin{array}{c}V_{n+2}+(1-r)V_{n+1}+tV_{n}\\
+(r-s-1)V_{0}+(r-1)V_{1}-V_{2}\end{array} \right\rbrace,
\end{equation}
where $\delta=\delta_{r,s,t}=r+s+t-1$ and $V_{n}$ denote the $n$-th term of the generalized Tribonacci numbers.
\end{lemma}

\begin{theorem}\label{n2}
The summation formula for generalized Tribonacci octonions is as follows:
\begin{equation}\label{p3}
\sum_{l=0}^{n}O_{v,l}=\frac{1}{\delta_{r,s,t}}(O_{v,n+2}+(1-r)O_{v,n+1}+tO_{v,n}+\omega_{r,s,t}),
\end{equation}
where $\omega_{r,s,t}=\lambda_{r,s,t}+e_{1}(\lambda_{r,s,t}-\delta V_{0})+\cdots + e_{7} (\lambda_{r,s,t}-\delta (V_{0}+\cdots +V_{6}))$, $\lambda_{r, s, t} = (r+s-1)V_{0}+(r-1)V_{1}-V_{2}$ and $\delta=\delta_{r,s,t}=r+s+t-1$.
\end{theorem}
\begin{proof}
Using Eq. (\ref{eq:2}), we have
\begin{align*}
\sum_{l=0}^{n}O_{v,l}&=\sum_{l=0}^{n}V_{l}+e_{1}\sum_{l=0}^{n}V_{l+1}+e_{2}\sum_{l=0}^{n}V_{l+2}+\cdots+e_{7}\sum_{l=0}^{n}V_{l+7}\\
&=(V_{0}+\cdots +V_{n})+e_{1}(V_{1}+\cdots +V_{n+1})\\
&\ \ +e_{2}(V_{2}+\cdots +V_{n+2})+\cdots +e_{7}(V_{7}+\cdots +V_{n+7}).
\end{align*}
Since from Eq. (\ref{p2}) and using the notation $\lambda_{r, s, t} = (r+s-1)V_{0}+(r-1)V_{1}-V_{2}$, we can write
\begin{align*}
\delta_{r,s,t}\sum_{l=0}^{n}O_{v,l}&=V_{n+2}+(1-r)V_{n+1}+tV_{n}+\lambda_{r, s, t} \\
&\ \ +e_{1}(V_{n+3}+(1-r)V_{n+2}+tV_{n+1}+\lambda_{r, s, t}-\delta V_{0})\\
&\ \ \vdots \\
&\ \ +e_{7}(V_{n+9}+(1-r)V_{n+8}+tV_{n+7}+\lambda_{r, s, t}-\delta (V_{0}+\cdots +V_{6}))\\
&=O_{v,n+2}+(1-r)O_{v,n+1}+tO_{v,n}+\omega_{r,s,t},
\end{align*}
where $\omega_{r,s,t}=\lambda_{r,s,t}+e_{1}(\lambda_{r,s,t}-\delta V_{0})+\cdots + e_{7} (\lambda_{r,s,t}-\delta (V_{0}+\cdots +V_{6}))$. Finally, $$\sum_{l=0}^{n}O_{v,l}=\frac{1}{\delta_{r,s,t}}(O_{v,n+2}+(1-r)O_{v,n+1}+tO_{v,n}+\omega_{r,s,t}).$$ The theorem is proved.
\end{proof}

The summation formula in Eq. (\ref{p3}) gives the sum of the elements in the octonion sequences which have not been found in the studies conducted so far. This can be seen in Table \ref{table:3}.
\begin{table}[ht] 
\caption{Summation formulas according to initial values.} 
\centering      
\begin{tabular}{ll}
\hline
$\textrm{Narayana}$ & $O_{v,n+3}-\left\lbrace \begin{array}{c}
1+e_{1}+2e_{2}+3e_{3}\\
+4e_{4}+6e_{5}+9e_{6}+13e_{7}\end{array} \right\rbrace$ \\ 
$\textrm{Tribonacci}$  & $\frac{1}{2}\left(O_{v,n+2}+O_{v,n}-\left\lbrace \begin{array}{c}
1+e_{1}+3e_{2}\\
+5e_{3}+9e_{4}+17e_{5}\\
+31e_{6}+57e_{7}\end{array} \right\rbrace \right)$ \\ 
$\textrm{Padovan}$  & $O_{v,n+5}-\left\lbrace \begin{array}{c}
1+e_{1}+2e_{2}+2e_{3}\\
+3e_{4}+4e_{5}+5e_{6}+7e_{7}\end{array} \right\rbrace $ \\ 
$\textrm{Third-order Jacobsthal}$  & $\frac{1}{3}\left(O_{v,n+2}+2O_{v,n}-\left\lbrace \begin{array}{c}
1+e_{1}+4e_{2}\\
+7e_{3}+13e_{4}+28e_{5}\\
+55e_{6}+109e_{7}\end{array} \right\rbrace \right)$ \\ 
\\ \hline
\end{tabular}
\label{table:3}  
\end{table}

Now, we present the formula which gives the norms for generalized Tribonacci octonions.
\begin{theorem}\label{n3}
The norm value for generalized Tribonacci octonions is given with the following formula:
\begin{equation}\label{p4}
Nr^{2}(O_{v,n})=\frac{1}{\phi^{2}}\left\lbrace \begin{array}{c}
(\omega_{1}-\omega_{2})^{2}P^{2}\overline{\alpha}\alpha^{2n}+(\alpha-\omega_{2})^{2}Q^{2}\overline{\omega_{1}}\omega_{1}^{2n}\\
+ (\alpha-\omega_{1})^{2}R^{2}\overline{\omega_{2}}\omega_{2}^{2n} -2K\end{array} \right\rbrace ,
\end{equation}
where $K=(\omega_{1}-\omega_{2})(\alpha-\omega_{2})PQ\underline{\alpha \omega_{1}}(\alpha \omega_{1})^{n}+(\omega_{1}-\omega_{2})(\alpha-\omega_{1})PR\underline{\alpha \omega_{2}}(\alpha \omega_{2})^{n}+(\alpha-\omega_{1})(\alpha-\omega_{2})QR\underline{\omega_{1} \omega_{2}}(\omega_{1} \omega_{2})^{n}$.
\end{theorem}
\begin{proof}
If we use the definition norm, then we obtain $Nr^{2}(O_{v,n})=\sum_{l=0}^{7}V_{n+l}^{2}$. Moreover, by the Binet formula (\ref{equ:5}) we have $$\phi V_{n}=(\omega_{1}-\omega_{2})P\alpha^{n}-(\alpha-\omega_{2})Q\omega_{1}^{n}+(\alpha-\omega_{1})R\omega_{2}^{n},$$ where $\phi =\phi(\alpha,\omega_{1},\omega_{2})=(\alpha-\omega_{1})(\alpha-\omega_{2})(\omega_{1}-\omega_{2})$. Then,
\begin{align*}
\phi^{2} V_{n}^{2}&=(\omega_{1}-\omega_{2})^{2}P^{2}\alpha^{2n}+(\alpha-\omega_{2})^{2}Q^{2}\omega_{1}^{2n}+(\alpha-\omega_{1})^{2}R^{2}\omega_{2}^{2n}\\
&\ \ -2(\omega_{1}-\omega_{2})(\alpha-\omega_{2})PQ(\alpha \omega_{1})^{n}+2(\omega_{1}-\omega_{2})(\alpha-\omega_{1})PR(\alpha \omega_{2})^{n}\\
&\ \ -2(\alpha-\omega_{1})(\alpha-\omega_{2})QR(\omega_{1} \omega_{2})^{n}
\end{align*}
and
\begin{align*}
\phi^{2} Nr^{2}(O_{v,n})&=\phi^{2} (V_{n}^{2}+V_{n+1}^{2}+\cdots +V_{n+7}^{2})\\
&=(\omega_{1}-\omega_{2})^{2}P^{2}\overline{\alpha}\alpha^{2n}+(\alpha-\omega_{2})^{2}Q^{2}\overline{\omega_{1}}\omega_{1}^{2n}\\
&\ \ + (\alpha-\omega_{1})^{2}R^{2}\overline{\omega_{2}}\omega_{2}^{2n}-2(\omega_{1}-\omega_{2})(\alpha-\omega_{2})PQ\underline{\alpha \omega_{1}}(\alpha \omega_{1})^{n}\\
&\ \ + 2(\omega_{1}-\omega_{2})(\alpha-\omega_{1})PR\underline{\alpha \omega_{2}}(\alpha \omega_{2})^{n}\\
&\ \ - 2(\alpha-\omega_{1})(\alpha-\omega_{2})QR\underline{\omega_{1} \omega_{2}}(\omega_{1} \omega_{2})^{n},
\end{align*}
where $\overline{\alpha}=1+\alpha^{2}+\alpha^{4}+\cdots + \alpha^{14}$, $\overline{\omega_{1,2}}=1+\omega_{1,2}^{2}+\omega_{1,2}^{4}+\cdots + \omega_{1,2}^{14}$, $\underline{\alpha \omega_{1,2}}=1+\alpha \omega_{1,2}+(\alpha \omega_{1,2})^{2}+\cdots+ (\alpha \omega_{1,2})^{7}$ and $\underline{\omega_{1} \omega_{2}}=1+\omega_{1} \omega_{2}+(\omega_{1} \omega_{2})^{2}+\cdots+ (\omega_{1} \omega_{2})^{7}$.
\end{proof}

\begin{theorem}\label{n4}
For $n\geq 0$, $m\geq 3$ we have
\begin{equation}\label{p5}
O_{v,n+m}=U_{m-1}O_{v,n+2}+(sU_{m-2}+tU_{m-3})O_{v,n+1}+tU_{m-2}O_{v,n},
\end{equation}
where $U_{n}=V_{n}(0,1,r;r,s,t)$.
\end{theorem}
\begin{proof}
For $m=3$, we have
\begin{align*}
O_{v,n+3}&=rO_{v,n+2}+sO_{v,n+1}+tO_{v,n}\\
&=U_{2}O_{v,n+2}+(sU_{1}+tU_{0})O_{v,n+1}+tU_{1}O_{v,n}.
\end{align*}
Suppose the equality holds for $m\leq l$. For $m=l+1$, we have
\begin{align*}
O_{v,n+l+1}&=rO_{v,n+l}+sO_{v,n+l-1}+tO_{v,n+l-2}\\
&=r \left(U_{l-1}O_{v,n+2}+(sU_{l-2}+tU_{l-3})O_{v,n+1}+tU_{l-2}O_{v,n}\right)\\
&\ \ + s \left(U_{l-2}O_{v,n+2}+(sU_{l-3}+tU_{l-4})O_{v,n+1}+tU_{l-3}O_{v,n}\right)\\
&\ \ + t \left(U_{l-3}O_{v,n+2}+(sU_{l-4}+tU_{l-5})O_{v,n+1}+tU_{l-4}O_{v,n}\right)\\
&=(rU_{l-1}+sU_{l-2}+tU_{l-3})O_{v,n+2}\\
&\ + \left(s(rU_{l-2}+sU_{l-3}+tU_{l-4})+t(rU_{l-3}+sU_{l-4}+tU_{l-5})\right)O_{v,n+1}\\
&\ + t(rU_{l-2}+sU_{l-3}+tU_{l-4})O_{v,n}\\
&=U_{l}O_{v,n+2}+(sU_{l-1}+tU_{l-2})O_{v,n+1}+tU_{l-1}O_{v,n}.
\end{align*}
By induction on $m$, we get the result.
\end{proof}
\begin{table}[ht] 
\caption{Convolution formulas $O_{v,n+m}$ according to initial values.} 
\centering      
\begin{tabular}{ll}
\hline
$\textrm{Narayana}$ & $
\left\lbrace \begin{array}{c}N_{m-1}O_{N,n+2}+N_{m-3}O_{N,n+1}\\
+N_{m-2}O_{N,n}\end{array} \right\rbrace$ \\ 
$\textrm{Tribonacci}$  & $
\left\lbrace \begin{array}{c} T_{m-1}O_{T,n+2}+(T_{m-2}+T_{m-3})O_{T,n+1}\\
+T_{m-2}O_{T,n}\end{array} \right\rbrace $ \\ 
$\textrm{Padovan}$  & $
\left\lbrace \begin{array}{c} P_{m-1}O_{P,n+2}+P_{m}O_{P,n+1}\\
+P_{m-2}O_{v,n},\end{array} \right\rbrace $ \\ 
$\textrm{Third-order Jacobsthal}$  & $
\left\lbrace \begin{array}{c} J_{m-1}O_{J,n+2}+(J_{m-2}+2J_{m-3})O_{J,n+1}\\
+2J_{m-2}O_{J,n},\end{array} \right\rbrace $ \\ 
\\ \hline
\end{tabular}
\label{table:4}  
\end{table}

Now, we give the quadratic approximation of $\{O_{v,n}\}$.
\begin{theorem}\label{n5}
Let $\{O_{v,n}\}_{n\geq0}$, $\alpha$, $\omega_{1}$ and $\omega_{2}$ be as above. Then, we have for all integer $n\geq 0$
\begin{equation}\label{p6}
\begin{array}{c} \textrm{Quadratic} \\
\textrm{app. of $\{O_{v,n}\}$}\end{array}: \left\{
\begin{array}{c }
P\underline{\alpha}\alpha^{n+2}=\alpha^{2} O_{v,n+2}+\alpha(sO_{v,n+1}+tO_{v,n})+ tO_{v,n+1},\\
Q\underline{\omega_{1}}\omega_{1}^{n+2}=\omega_{1}^{2} O_{v,n+2}+\omega_{1}(sO_{v,n+1}+tO_{v,n})+ tO_{v,n+1},\\
R\underline{\omega_{2}}\omega_{2}^{n+2}=\omega_{2}^{2} O_{v,n+2}+\omega_{2}(sO_{v,n+1}+tO_{v,n})+ tO_{v,n+1},
\end{array}
\right.
\end{equation}
where $P=V_{2}-(\omega_{1}+\omega_{2})V_{1}+\omega_{1}\omega_{2}V_{0}$, $Q=V_{2}-(\alpha+\omega_{2})V_{1}+\alpha\omega_{2}V_{0}$ and $R=V_{2}-(\alpha+\omega_{1})V_{1}+\alpha\omega_{1}V_{0}$. Furthermore, $\underline{\alpha}=\sum_{l=0}^{7}\alpha^{l}e_{l}$ and $\underline{\omega_{1,2}}=\sum_{l=0}^{7}\omega_{1,2}^{l}e_{l}$.
\end{theorem}
\begin{proof}
Using the Binet's formula Eq. (\ref{equ:5}), we have
\begin{align*}
\alpha O_{v,n+2}&+(s+\omega_{1}\omega_{2})O_{v,n+1}+tO_{v,n}\\
&=\frac{P\underline{\alpha}\alpha^{n}(\alpha^{3}+(s+\omega_{1}\omega_{2})\alpha+t)}{(\alpha-\omega_{1})(\alpha-\omega_{2})}-\frac{Q\underline{\omega_{1}}\omega_{1}^{n}(\alpha\omega_{1}^{2}+(s+\omega_{1}\omega_{2})\omega_{1}+t)}{(\alpha-\omega_{1})(\omega_{1}-\omega_{2})}\\
&\ \ +\frac{R\underline{\omega_{2}}\omega_{2}^{n}(\alpha\omega_{2}^{2}+(s+\omega_{1}\omega_{2})\omega_{2}+t)}{(\alpha-\omega_{2})(\omega_{1}-\omega_{2})}\\
&=(V_{2}-(\omega_{1}+\omega_{2})V_{1}+\omega_{1}\omega_{2}V_{0})\underline{\alpha}\alpha^{n+1},
\end{align*}
the latter given that $\alpha\omega_{1}^{2}+(s+\omega_{1}\omega_{2})\omega_{1}+t=0$ and $\alpha\omega_{2}^{2}+(s+\omega_{1}\omega_{2})\omega_{2}+t=0$. Then, we get
\begin{equation}\label{eq:17}
\alpha O_{v,n+2}+(s+\omega_{1}\omega_{2})O_{v,n+1}+tO_{v,n}=(V_{2}-(\omega_{1}+\omega_{2})V_{1}+\omega_{1}\omega_{2}V_{0})\underline{\alpha}\alpha^{n+1}.
\end{equation}
Multiplying Eq. (\ref{eq:17}) by $\alpha$ and using $\alpha\omega_{1}\omega_{2}=t$, we have 
\begin{align*}
P\underline{\alpha}\alpha^{n+2}&=\alpha^{2}O_{v,n+2}+\alpha(s+\omega_{1}\omega_{2})O_{v,n+1}+\alpha tO_{v,n}\\
&=\alpha^{2} O_{v,n+2}+\alpha(sO_{v,n+1}+tO_{v,n})+ tO_{v,n+1},
\end{align*}
where $P=V_{2}-(\omega_{1}+\omega_{2})V_{1}+\omega_{1}\omega_{2}V_{0}$. If we change $\alpha$, $\omega_{1}$ and $\omega_{2}$ role above process, we obtain the desired result Eq. (\ref{p6}).
\end{proof}

The next theorem gives an alternative proof of the Binet's formula for the generalized Tribonacci octonions (see Eq. (\ref{equ:5})).
\begin{theorem}
For any integer $n\geq 0$, the $n$-th generalized Tribonacci octonion is
\begin{equation}\label{binQ}
O_{v,n}=\frac{P\underline{\alpha}\alpha^{n}}{(\alpha-\omega_{1})(\alpha-\omega_{2})}-\frac{Q\underline{\omega_{1}}\omega_{1}^{n}}{(\alpha-\omega_{1})(\omega_{1}-\omega_{2})}+\frac{R\underline{\omega_{2}}\omega_{2}^{n}}{(\alpha-\omega_{2})(\omega_{1}-\omega_{2})},
\end{equation}
where $P$, $Q$ and $R$ as in Eq. (\ref{eq:8}), $\underline{\alpha}=\sum_{l=0}^{7}\alpha^{l}e_{l}$ and $\underline{\omega_{1,2}}=\sum_{l=0}^{7}\omega_{1,2}^{l}e_{l}$. If $V_{0}=V_{1}=0$, $V_{2}=1$ and $r=s=t=1$, we get the classic Tribonacci octonion. 
\end{theorem}
\begin{proof}
For the Eq. (\ref{p6}), we have
\begin{align*}
&\alpha^{2} Q_{V,n+2}+\alpha \left( sQ_{V,n+1}+tQ_{V,n}\right) +tQ_{V,n+1}\\
&=\alpha^{2} \left(V_{n+2}+V_{n+3}e_{1}+\cdots+V_{n+5}e_{7}\right)\\
&\ \ +\alpha \left(sV_{n+1}+tV_{n} +(sV_{n+2}+tV_{n+1})e_{1}+\cdots +(sV_{n+4}+tV_{n+3})e_{7}\right)\\
&\ \ +t \left(V_{n+1}+V_{n+2}e_{1}+\cdots +V_{n+4}e_{7}\right)\\
&=\alpha^{2} V_{n+2}+\alpha \left( sV_{n+1}+tV_{n}\right) +tV_{n+1}+\left(\alpha^{2} V_{n+3}+\alpha \left( sV_{n+2}+tV_{n+1}\right) +tV_{n+2}\right)e_{1}\\
&\ \ + \left(\alpha^{2} V_{n+4}+\alpha \left( sV_{n+3}+tV_{n+2}\right) +tV_{n+3}\right)e_{2}\\
&\ \ \vdots \\
&\ \ +\left(\alpha^{2} V_{n+5}+\alpha \left( sV_{n+4}+tV_{n+3}\right) +tV_{n+4}\right)e_{7}.
\end{align*}
From the identity $P\alpha^{n+2}=\alpha^{2} V_{n+2}+\alpha(sV_{n+1}+tV_{n})+ tV_{n+1}$ for $n$-th generalized Tribonacci number $V_{n}$, we obtain
\begin{equation}\label{ge1}
\alpha^{2} O_{v,n+2}+\alpha \left( sO_{v,n+1}+tO_{v,n}\right) +tO_{v,n+1}=P\underline{\alpha}\alpha^{n+2}.
\end{equation}
Similarly, we have
\begin{equation}\label{ge2}
\omega_{1}^{2} O_{v,n+2}+\omega_{1} \left( sO_{v,n+1}+tO_{v,n}\right) +tO_{v,n+1}=Q\underline{\omega_{1}}\omega_{1}^{n+2},
\end{equation}
\begin{equation}\label{ge3}
\omega_{2}^{2} O_{v,n+2}+\omega_{2} \left( sO_{v,n+1}+tO_{v,n}\right) +tO_{v,n+1}=R\underline{\omega_{2}}\omega_{2}^{n+2}.
\end{equation}
Subtracting Eq. (\ref{ge2}) from Eq. (\ref{ge1}) gives 
\begin{equation}\label{d1}
(\alpha+\omega_{1}) O_{v,n+2}+\left( sO_{v,n+1}+tO_{v,n}\right) =\frac{P\underline{\alpha}\alpha^{n+2}-Q\underline{\omega_{1}}\omega_{1}^{n+2}}{\alpha-\omega_{1}}.
\end{equation}
Similarly, subtracting Eq. (\ref{ge3}) from Eq. (\ref{ge1}) gives 
\begin{equation}\label{d2}
(\alpha+\omega_{2}) O_{v,n+2}+\left( sO_{v,n+1}+tO_{v,n}\right) =\frac{P\underline{\alpha}\alpha^{n+2}-R\underline{\omega_{2}}\omega_{2}^{n+2}}{\alpha-\omega_{2}}.
\end{equation}
Finally, subtracting Eq. (\ref{d2}) from Eq. (\ref{d1}), we obtain 
\begin{align*}
O_{v,n+2}&=\frac{1}{\omega_{1}-\omega_{2}}\left( \frac{P\underline{\alpha}\alpha^{n+2}-Q\underline{\omega_{1}}\omega_{1}^{n+2}}{\alpha-\omega_{1}}-\frac{P\underline{\alpha}\alpha^{n+2}-R\underline{\omega_{2}}\omega_{2}^{n+2}}{\alpha-\omega_{2}}\right)\\
&=\frac{P\underline{\alpha}\alpha^{n+2}}{(\alpha-\omega_{1})(\alpha-\omega_{2})}-\frac{Q\underline{\omega_{1}}\omega_{1}^{n+2}}{(\alpha-\omega_{1})(\omega_{1}-\omega_{2})}+\frac{R\underline{\omega_{2}}\omega_{2}^{n+2}}{(\alpha-\omega_{2})(\omega_{1}-\omega_{2})}.
\end{align*}
 So, the theorem is proved.
\end{proof}

\section{Conclusions}
Octonions have great importance as they are used in quantum physics, applied mathematics, graph theory. In this work, we introduce the generalized Tribonacci octonion numbers and formulate the Binet-style formula, the generating function and some identities of the
generalized Tribonacci octonion sequence. Thus, in our future studies we plan to examine different quaternion and octonion polynomials and their key features.



\begin{thebibliography}{9}
\bibitem{Ad}
S.L. Adler, Quaternionic quantum mechanics and quantum fields, \emph{New York: Oxford University Press}, 1994.
\bibitem{Ak}
I. Akkus and O. Ke\c{c}ilioglu, Split Fibonacci and Lucas octonions, \emph{Adv. Appl. Clifford Algebras}, \textbf{25(3)}, (2015), 517--525.
\bibitem{Aky}
M. Akyigit, H.H. K\"osal and M. Tosun, Split Fibonacci quaternions, \emph{Adv. Appl. Clifford Algebras}, \textbf{23}, (2013), 535--545.
\bibitem{Bae} 
J.C. Baez, The octonions, \emph{Bull. Am. Math. Soc.}, \textbf{39}, (2002), 145--205.
\bibitem{Car}
K. Carmody, Circular and Hyperbolic Quaternions, Octonions and Sedenions, \emph{Appl. Math. Comput.}, \textbf{28}, (1988), 47--72.
\bibitem{Ca} 
P. Catarino, The modified Pell and the modified $k$-Pell quaternions and octonions, \emph{Adv. Appl. Clifford Algebras}, \textbf{26},  (2016), 577--590.
\bibitem{Cer} 
G. Cerda-Morales, Identities for Third Order Jacobsthal Quaternions, \emph{Advances in Applied Clifford Algebras}, \textbf{27(2)}, (2017), 1043--1053.
\bibitem{Cer1} 
G. Cerda-Morales, On a Generalization of Tribonacci Quaternions, \emph{Mediterranean Journal of Mathematics}, \textbf{14:239} (2017), 1--12.
\bibitem{Cim1} 
C.B. \c{C}imen and A. \.{I}pek, On Pell quaternions and Pell-Lucas quaternions, \emph{Adv. Appl. Clifford Algebras}, \textbf{26(1)}, (2016), 39--51.
\bibitem{Cim2} 
C.B. \c{C}imen and A. \.{I}pek, On Jacobsthal and Jacobsthal-Lucas Octonions, \emph{Mediterranean Journal of Mathematics}, \textbf{14:37}, (2017), 1--13.
\bibitem{Fe}
M. Feinberg, Fibonacci-Tribonacci, \emph{The Fibonacci Quarterly}, \textbf{1(3)}, (1963), 71--74.
\bibitem{Ge}
W. Gerdes, Generalized Tribonacci numbers and their convergent sequences, \emph{The Fibonacci Quarterly}, \textbf{16(3)}, (1978), 269--275.
\bibitem{Go1} 
M. Gogberashvili, Octonionic Geometry, \emph{Adv. Appl. Clifford Algebras}, \textbf{15}, (2005), 55--66.
\bibitem{Go2} 
M. Gogberashvili, Octonionic electrodynamics, \emph{J. Phys. A: Math. Gen.}, \textbf{39}, (2006), 7099--7104.
\bibitem{Hal1} 
S. Halici, On Fibonacci quaternions, \emph{Adv. Appl. Clifford Algebras}, \textbf{22}, (2012), 321--327.
\bibitem{Hal2} 
S. Halici, On complex Fibonacci quaternions, \emph{Adv. Appl. Clifford Algebras}, \textbf{23}, (2013), 105--112.
\bibitem{Hor1} 
A. F. Horadam, Complex Fibonacci numbers and Fibonacci quaternions, \emph{Am. Math. Month.}, \textbf{70}, (1963), 289--291.
\bibitem{Hor2} 
A. F. Horadam, Quaternion recurrence relations, \emph{Ulam Quarterly}, \textbf{2}, (1993), 23--33 .
\bibitem{Iye} 
M.R. Iyer, A note on Fibonacci quaternions, \emph{Fibonacci Quaterly}, \textbf{7(3)}, (1969), 225--229.
\bibitem{Ke-Ak} 
O. Ke\c{c}ilio\u{g}lu and I. Akkus, The Fibonacci Octonions, \emph{Adv. Appl. Clifford Algebras}, \textbf{25(1)}, (2015), 151--158.
\bibitem{Ki}
E. Kili\c{c}, Tribonacci sequences with certain indices and their sums, \emph{Ars Comb.}, \textbf{86}, (2008), 13--22.
\bibitem{Ko1}
J. K\"oplinger, Signature of gravity in conic sedenions, \emph{Appl. Math. Computation}, \textbf{188}, (2007), 942--947.
\bibitem{Ko2}
J. K\"oplinger, Hypernumbers and relativity, \emph{Appl. Math. Computation}, \textbf{188}, (2007), 954--969.
\bibitem{Pe}
S. Pethe, Some identities for Tribonacci sequences, \emph{The Fibonacci Quarterly}, \textbf{26(2)}, (1988), 144--151.
\bibitem{Sha}
A.G. Shannon and A.F. Horadam, Some properties of third-order recurrence relations, \emph{The Fibonacci Quarterly}, \textbf{10(2)}, (1972), 135--146.
\bibitem{Spi}
W. R. Spickerman, Binet's formula for the Tribonacci numbers, \emph{The Fibonacci Quarterly}, \textbf{20}, (1982), 118--120.
\bibitem{Szy-Wl} 
A. Szynal-Liana and I. W\l och, A Note on Jacobsthal Quaternions, \emph{Adv. Appl. Clifford Algebras}, \textbf{26}, (2016), 441--447.
\bibitem{Ta}
Y. Tian, Matrix representations of octonions and their applications, \emph{Adv. Appl. Clifford Algebras}, \textbf{10(1)}, (2000), 61--90.
\bibitem{Ya}
C.C. Yalavigi, Properties of Tribonacci numbers, \emph{The Fibonacci Quarterly}, \textbf{10(3)}, (1972), 231--246.

\end{thebibliography}
\end{document}